\documentclass[reqno,12 pt,letterpaper]{amsart}
\usepackage[margin=1.0in]{geometry}
\usepackage{graphicx}
\usepackage{amsmath, amssymb, amsthm}

\usepackage[colorinlistoftodos]{todonotes}

\usepackage{fancyhdr}
\usepackage{subfigure}

\newtheorem{theorem}{Theorem}[section]
\newtheorem{lemma}[theorem]{Lemma}

\theoremstyle{definition}
\newtheorem{definition}[theorem]{Definition}

\newtheorem{fact}[theorem]{Fact}

\theoremstyle{remark}
\newtheorem{remark}[theorem]{Remark}

\numberwithin{equation}{section}

\newcommand{\abs}[1]{\lvert#1\rvert}

\newcommand{\diam}[0]{\normalfont \text{diam}}
\newcommand{\spn}[0]{\normalfont \text{span}}
\newcommand{\rn}[0]{\normalfont \text{rn}}
\newcommand{\tn}[0]{\normalfont \text{tn}}
\newcommand{\rl}[0]{\normalfont \text{RL}}

\begin{document}

\title{The Radio Number of Grid Graphs}

\author{Tian-Shun Allan Jiang}
\dedicatory{
The North Carolina School of Science and Mathematics}
\date{January 6, 2014}






\begin{abstract}
\begingroup
\fontsize{12pt}{12pt}\selectfont
The radio number problem uses a graph-theoretical model to simulate optimal frequency assignments on wireless networks. A radio labeling of a connected graph $G$ is a function $f:V(G) \to \mathbb Z_{0}^+$ 
such that for every pair of vertices $u,v \in V(G)$, we have $\abs{f(u)-f(v)} \ge \diam(G) + 1 - d(u,v)$ where $\diam(G)$ denotes the diameter of $G$ and $d(u,v)$ the distance between vertices $u$ and $v$. Let $\spn(f)$ be the difference between the greatest label and least label assigned to $V(G)$. Then, the \textit{radio number} of a graph $\rn(G)$ is defined as the minimum value of $\spn(f)$ over all radio labelings of $G$. So far, there have been few results on the radio number of the grid graph: In 2009 Calles and Gomez 
gave an upper and lower bound for square grids, and in 2008 Flores and Lewis 
were unable to completely determine the radio number of the ladder graph (a 2 by $n$ grid).
In this paper, we completely determine the radio number of the grid graph $G_{a,b}$ for $a,b>2$, characterizing three subcases of the problem and providing a closed-form solution to each. These results have implications in the optimization of radio frequency assignment in wireless networks such as cell towers and environmental sensors.
\endgroup
\end{abstract}
\maketitle

\newpage

\section{Introduction}\label{sec:intro}
\subsection{Motivation} Wireless communication pervades modern society, but many of us have experienced problems with wireless networks due to interference -- fuzzy radio reception, choppy cellphone calls, or poor Wi-Fi connections. As the radio spectrum becomes more crowded due to the proliferation of wireless technologies, network planners are searching for more efficient ways to allocate broadcast frequencies. An effective method of frequency coordination mitigates inter-channel interference to maximize the efficiency of wireless communication. 
The efficient usage of wireless frequency spectra helps reduce crowding on heavily used frequency bands such as the 2.4 GHz band, which is used for wireless internet, cell phones, microwaves, and Bluetooth \cite{wireless-phone-internet, wireless-microwave-bluetooth}.
To approach this problem, we used a graph-theoretical model known as  radio labeling to optimize wireless frequency assignment. 
In this study, we research grid-like networks and arrive at a complete solution to the radio number of grid graphs.

\subsection{Background}

Radio labeling is motivated by the channel assignment problem introduced by Hale \cite{hale original problem}, who defined a whole class of graph-labeling problems designed to optimize wireless frequency assignment. Suppose a wireless network is composed of a set of radio stations, all separated by some geographic distance. If we treat each radio station as a vertex and assign an edge between two stations if they are geographically ``close" to each other, then the radio labeling of the graph is closely related to the network's optimal frequency assignment \cite{hale original problem}. 
Finding the radio number of a network's graph is theoretically equivalent to minimizing 
the spectral range of required frequencies used by the network. 

Radio labelings are a special case of the $L(h,k)$-labeling problem, which has been studied extensively in the past decade \cite{L(2-1),a,c,b,e,f,g} with many different approaches, including graph theory and combinatorics, simulated annealing \cite{simulated-annealing2,simulated-annealing1}, genetic algorithms \cite{genetic-algorithms}, tabu search \cite{tabu-search}, and neural networks \cite{neural-network}. Previously, Calamoneri \cite{L(h-k)grids} studied the $L(h,k)$ labeling on grid graphs, but was unable to solve the radio number subcase. The problem is also closely related to the upper traceable number \cite{NP-results}, a type of metric travelling salesman problem (TSP) which is an NP-hard problem \cite{TSP-hardness} in combinatorial optimization with applications in operations research \cite{tsp-operations-research} and theoretical computer science \cite{tsp-cs}.

\subsection{Definitions}
A \textit{radio labeling} labels the vertices of a graph with nonnegative integers such that for any two vertices, the smaller the distance between the vertices, the greater the required difference in label. Then, the \textit{radio number} of a graph is the value of the smallest possible span of integer labels on a radio labeling. 
Formally, let $G$ be a simple connected graph with vertex set $V(G)$, let $d(u,v)$ denote the distance\footnote{Also called the \textit{geodesic distance}, the distance between two vertices of a graph is number of edges in the shortest path connecting them.} between any pair of distinct vertices $u,v \in V(G)$, and let $\diam(G)$ denote the diameter of $G$. Then Li \cite{m-ary trees} gives the following definition:
A radio labeling of a connected graph $G$ is a mapping $f : V(G) \to \mathbb Z_{0}^+$ such that $\abs{f(u) - f(v)} \ge \diam(G) +1 - d(u,v)$ for each pair of distinct vertices $u,v \in V(G)$. The span of $f$ is defined as $\max_{u,v \in V(G)} \abs{f(u) - f(v)}$, and the radio number of $G$, $\rn(G)$, is the minimum span of a radio labeling of $G$.

\subsection{Previous Work}
Since its introduction by Hale in 1980, Chartrand et al. \cite{chartrand-orig, chartrand-final, chartrand-cycle} first studied the radio number of path and cycle graphs, giving upper and lower bounds for their values. Later in 2005, Liu and Zhu \cite{pathcycle} determined their exact values. In the following years, Liu \cite{trees} found upper and lower bounds for the radio number of trees, Khennoufa and Togni \cite{hypercubes} used binary Gray codes to determine the radio number of hypercubes, and Li et al. \cite{m-ary trees} determined the radio number of m-ary trees. Since it is computationally complex to calculate the radio number on general graphs (known to be NP-complete for certain graphs $G$ \cite{NP-results}), we direct our attention to the radio number on special classes of graphs. We chose to investigate the grid graph for its similarity to real-world applications. Before this work, Calles and Gomez \cite{unpublished-grid} had the most recent progress on this problem, establishing upper and lower bounds for $\rn(G)$ for grids of equal $x$ and $y$ dimensions.


\subsection{Our Results}
In this paper, we completely determine $\rn(G)$ on any grid graph\footnote{Formally defined as the Cartesian product between two path graphs of length $a$ and $b$.} $G_{a,b}$ where $a,b>2$.

In Section \ref{sec:prelim} of this paper, we introduce some definitions and fundamental observations.
In Sections \ref{sec:utn} and \ref{sec:bump}, we investigate the \textit{upper traceable number} and \textit{bump} on grid graph $G_{a,b}$, two important metrics in determining $\rn(G_{a,b})$. In Section \ref{sec:grid}, we determine the value of $\rn(G_{a,b})$ with $a,b > 2$.

Our methodology not only yields important results on the grid graph, but also represents an improved approach to the problem by simplifying the problem to an equivalent one of finding appropriate vertex orderings $s^*$.



\section{Preliminaries}\label{sec:prelim}
For ease of notation, we refer to $\diam(G)$ as $D$, the distance between vertices $u$ and $v$ as $d(u, v)$, and a radio labeling on $G$ as $f$. We begin by observing the following fact:

\begin{fact}\label{fact:no repeat}
In a radio labeling, no two vertices may have the same label since $d(u,v) \le D$.
\end{fact}

By Fact \ref{fact:no repeat}, given any radio labeling $f$, there exists a unique ordering of vertices by increasing label. Let this unique ordering associated with $f$ be $s_f: u_{1}, u_{2}, \ldots u_{n}$, where $n = \abs{V(G)}$
, such that $$f(u_1)<f(u_2)<\ldots<f(u_n).$$ 
Notice that there exist multiple $f$'s with ordering $s_f$ (i.e. there is a many-to-one relationship between $f$ and $s$). 
Conversely, when given an ordering $s$, we are interested in the set of corresponding $f$ with minimum span. 

For convenience, let $f_i = f(u_{i+1}) - f(u_{i})$ and $d_i = d(u_{i+1},u_{i})$. Then, by the definition of a radio labeling we have
$$f_i \ge D + 1 - d_i.$$ 
By taking the sum of all inequalities for $1 \le i \le n-1$, we have
$$\spn(f) = \sum_{i=1}^{n-1}f_i \ge (n-1)(D+1) - \sum_{i=1}^{n-1}d_i.$$
Now, let $\rl(G)$ be the set of radio labelings on $G$.
Since the radio number is the minimum span of a labeling, we get the following Lemma by Liu \cite{trees}:
\begin{lemma}
The lower bound of the radio number on any connected graph $G$ is 
$$\rn(G) \ge (n-1)(D+1) - \max_{f \in \rl(G)}\left(\sum_{i=1}^{n-1} d_i \right).$$
\end{lemma}

To simplify notation, the $\max$ and $\min$ functions are always maximizing or minimizing an expression over the set $\rl(G)$.
For completeness, we also present an upper bound on the radio number on any connected graph $G$ in Lemma \ref{lem:upper bound}.
\begin{lemma}\label{lem:upper bound}
$\rn(G) = O(n^2)$. In particular,
$$\rn(G) \le D \cdot (n-1) \le (n-1)^2.$$
\end{lemma}

\begin{proof}
Since $D \le n-1 \implies D \cdot (n-1) \le (n-1)^2$, it remains to be shown that $\rn(G) \le D \cdot (n-1).$

If the vertices of $G$ are in any order $v_{1},v_{2},\ldots,v_{n}$, let $f(v_{i})=D \cdot (i-1)$ for $1 \le i \le n-1$. This is a radio labeling it has span $D \cdot (n-1)$ as specified.
\end{proof}

We now introduce some terminology to more accurately describe $\rn(G)$:
\begin{definition}
As given by Li \cite{m-ary trees}, for two consecutive vertices $u_{i},u_{i+1}$ in an ordering $s$, define the \textit{bump} between $u_{i+1}$ and $u_{i}$ as 
$$b(u_{i+1},u_{i}) := f_i - (D + 1 - d_i).$$
\end{definition}

For convenience, we denote $b(u_{i+1},u_{i})$ by $b_i$.

\begin{definition}\label{tightness neighbor}
A vertex $u_i$'s set of \textit{tightness neighbors} $\tn(u_i)$ consists of all vertices $u_j$ such that 
$$\abs{f(u_i)-f(u_j)} = D+1 - d(u_i,u_j).$$
\end{definition}

\begin{remark}\label{rem:additional value}
Notice that $f_i = D+1 - d_i + b_i$. This shows that $b_i > 0$ if and only if  $u_{i-1} \not\in \tn(u_i)$. 
\end{remark}

Using Remark \ref{rem:additional value} and summing up over $1 \le i \le n-1$, we get
\begin{lemma}\label{lem:rn is d-b}
Given a simple, connected graph $G$,
$$\rn(G) = (n-1)(D+1) - \max\left(\sum_{i=1}^{n-1}\left(d_i-b_i\right)\right) \ge (n-1)(D+1) -  \left(\max\sum_{i=1}^{n-1}d_i - \min\sum_{i=1}^{n-1}b_i\right).$$
\end{lemma}

To find $\max\left(\sum_{i=1}^{n-1}\left(d_i-b_i\right)\right)$, we study $\max\left(\sum_{i=1}^{n-1}d_i\right)$, an NP-hard problem \cite{NP-hard-utn} known as the \textit{upper traceable number} $t^+(G)$ in Section \ref{sec:utn} and the total bump $\min\left(\sum_{i=1}^{n-1}b_i\right)$ in Section \ref{sec:bump}.


\section{Upper Traceable Number of Grid Graphs}\label{sec:utn}
Let $G_{a,b}$ denote the grid graph with $a$ vertices in the $x$-direction and $b$ vertices in the $y$-direction.
For ease of notation, given an ordering $s$, let $d(s) = \sum_{i=1}^{n-1}d_i$. We establish the following theorem on $t^+(G_{a,b})$:

\begin{theorem}\label{thm:utn grid}
On the grid graph $G_{a,b}$ with $a,b>2$,
$$
t^+(G_{a,b}) =
\begin{cases}
\dfrac{a^2b+b^2a}{2}-3, & \normalfont\text{if }a,b\text{ are even, }
\\[1em]
\dfrac{a^2b+b^2a-a}{2}-1, & \normalfont\text{if }a\text{ is even, }b \text{ is odd}
\\[1em]
\dfrac{a^2b+b^2a-a-b}{2}-1, & \normalfont\text{if }a,b\text{ are odd, }
\end{cases}
$$
Clearly, the case $a$-even $b$-odd is equivalent to $a$-odd $b$-even if we reorient the grid $G_{a,b}$ to $G_{b,a}$.
\end{theorem}

We prove the theorem using the following Lemma:
\begin{lemma}\label{lem:max x distance}
Let $d_x(s)$ denote the sum of the $x$-components of each distance $d_i$ in ordering $s$. Then,
$$ 
\max(d_x(s)) =
\begin{cases}
\dfrac{a^2b}{2}-1, & \normalfont\text{if }a\text{ is even}\\[1em]
\dfrac{(a^2-1)b}{2}, & \normalfont\text{if }a\text{ is odd}
\end{cases}
$$
\end{lemma}

\begin{proof}
First, we construct an ordering $s^*$ that attains the value of $\max(d_x(s))$ given in Lemma \ref{lem:max x distance} above (See Figure \ref{fig:max x distance}). Denote the $x$-value of a vertex $u_i$ by $x_i$.

\begin{figure}[htb]
    \centering
    \hspace*{\fill}
    \subfigure[Example of regions A and B on $G_{6,4}$. No two consecutive vertices are both in A or both in B.]
    {
        \includegraphics[height=1.6in]{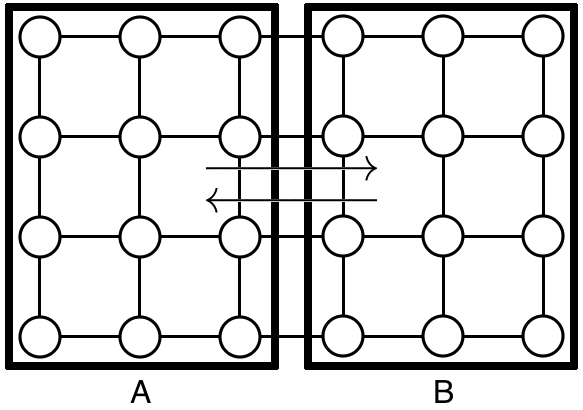}
        \label{fig:max x even}
    }
    \hspace*{\fill}
    \subfigure[Example of regions A, B and X on $G_{5,4}$. No two consecutive vertices are both in A or both in B. However, two consecutive vertices may be in X.]
    {
        \includegraphics[height=1.6in]{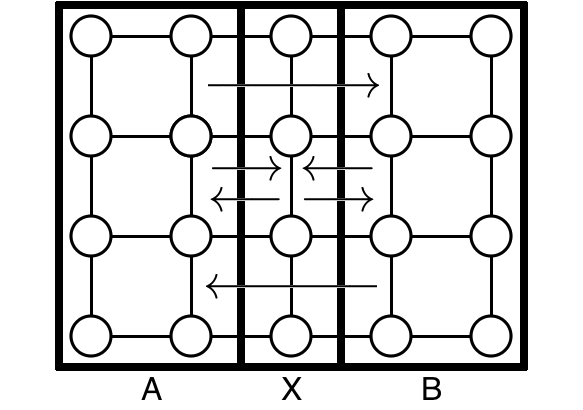}
        \label{fig:max x odd}
    }
    \hspace*{\fill}
    \caption{Constructing $s^*$ that achieves $\max(d_x(s))$.}
    \label{fig:max x distance}
\end{figure}

\textbf{Case 1: a is even}
\begin{enumerate}
\item $x_1 = \frac{a}{2}$ and $x_n = \frac{a}{2}+1$.
\item Each successive $u_{i+1}$ alternates from Region A to Region B, where 
\subitem Region A consists of all vertices such that $x_j \le \frac{a}{2}$, and
\subitem Region B consists of all vertices such that $x_k \ge \frac{a}{2}+1$. See Figure \ref{fig:max x even}.
\end{enumerate}

\textbf{Case 2: a is odd}
\begin{enumerate}
\item $x_1 = x_n = \frac{a+1}{2}$.
\item There does not exist a pair $u_i, u_{i+1}$ which are both in Region A or both in Region B, where
\subitem Region A consists of all vertices such that $x_j < \frac{a+1}{2}$, and
\subitem Region B consists of all vertices such that $x_k > \frac{a+1}{2}$, and
\subitem Region X consists of all vertices such that $x_l = \frac{a+1}{2}$. See Figure \ref{fig:max x odd}.
\end{enumerate}

Now, we establish an upper bound on the value of $d_x(s)$.
Since $d_x(u_{i+1},u_{i}) = \abs{x_{i+1}-x_{i}}$, one vertex contributes a positive value and the other a negative value to the distance calculation. Notice that we can represent distances by labeling the graph with $+$ and $-$ signs \cite{pathcycle}, as in Figure \ref{fig:+- distance}. Looking at $d_x(s) = \sum_{i=1}^{n-1}\abs{x_{i+1}-x_i}$, we see that the sum has $2$ of each $x_2$ through $x_{n-1}$, and 1 of each term $x_1$ and $x_n$. Half of the terms in the sum are positive, and the other half are negative. In fact, the value of $d_x(s)$ can be modeled as such: given multisets $A = \{a_1,a_2,\ldots,a_{2n}\}$ and $B = \{b_1,b_2,\ldots,b_{2n}\}$, where 
$$A = \{\overbrace{-1,\ldots,-1}^{n-1},0,0,\overbrace{1,\ldots,1}^{n-1}\} \text{ and } B = \{\overbrace{1,\ldots,1}^{2b},\overbrace{2,\ldots,2}^{2b},\ldots,\overbrace{a,\ldots,a}^{2b}\},$$
any value of $d_x(s)$ can be modeled by the quantity $a_1b_1+a_2b_2+\ldots+a_{2n}b_{2n}$. By the rearrangement inequality \cite{rearrangement-ineq}, the maximum sum of $d_x(s) = a_1b_1+a_2b_2+\ldots+a_{2n}b_{2n}$ occurs when $A$ and $B$ are both sorted in nondecreasing order, giving us an upper bound.

\begin{figure}[htb]
    \includegraphics[scale=0.4]{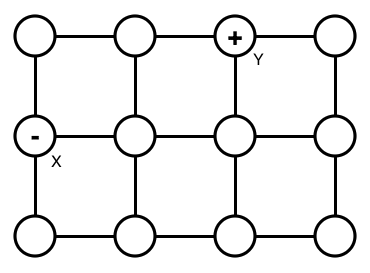}
    \caption{Using $+$ and $-$ signs to represent $x$-component distance. Above, we have $d_x(X,Y) = -1 + 3 = 2$}
    \label{fig:+- distance}
\end{figure}

When depicting this maximum sum using $+$ and $-$ signs on the grid graph, there are two configurations based on whether $a$ is even or odd.

\textit{When $a$ is even:}
\begin{itemize}
\item Region B contains all $+$ signs, and Region A contains all $-$ signs.
\item The two vertices with 1 sign have $x$-values of $\frac{a}{2}$ and $\frac{a}{2}+1$ respectively.
\end{itemize}

\textit{When $a$ is odd:}
\begin{itemize}
\item Region B contains all $+$ signs, and Region A contains all $-$ signs.
\item There are an equal number of $+$ and $-$ signs in Region X, and the two vertices with 1 label lie in Region X.
\end{itemize}

Notice that if we represent $d_x(s^*)$ with $+$ and $-$ signs (from the construction given above), the $+$ and $-$ sign representation coincides with the $+$ and $-$ signs in the upper bound given here. In fact, the $+$ and $-$ sign representation given here may only be achieved by the construction $s^*$ given above. Since we constructed cases that achieve the upper bound of $d_x(s)$, we have proved that our construction attains $\max(d_x(s))$ and we are done.
\end{proof}

Now that we have established orderings $s^*$ that attain $\max(d_x(s))$, we use these results in the $x$ and $y$ directions to prove Theorem \ref{thm:utn grid}.

\textit{Proof of Theorem \ref{thm:utn grid}.}

\textbf{Case A - $a,b$ are even:}
Using the fact that $t^+(G_{a,b}) = \max(d(s)) \le \max(d_x(s)) + \max(d_y(s))$  because $d(s) = d_x(s) + d_y(s)$, we show that $t^+(G_{a,b}) = \max(d_x(s)) + \max(d_y(s)) - 1$ for $a,b$ even. 

For sake of contradiction, assume $\max(d(s)) = \max(d_x(s)) + \max(d_y(s))$. Let 
the $1^\text{st}$ quadrant be the set of vertices with $x_j \geq \frac{a+1}{2}$ and $y_j \geq \frac{b+1}{2}$,
the $2^\text{nd}$ quadrant be the set of vertices with $x_j \leq \frac{a}{2}$ and $y_j \geq \frac{b+1}{2}$,
the $3^\text{rd}$ quadrant be the set of vertices with $x_j \leq \frac{a}{2}$ and $y_j \leq \frac{b}{2}$, and
the $4^\text{th}$ quadrant be the set of vertices with $x_j \geq \frac{a+1}{2}$ and $y_j \leq \frac{b}{2}$. 
See Figure \ref{fig:even x even quadrants} below. 
Then, applying the results from Lemma \ref{lem:max x distance}, we split the grid into 4 quadrants and consider the necessary conditions to simultaneously achieve $\max(d_x(s))$ and $\max(d_y(s))$.
\begin{enumerate}
\item Vertices $u_1$ and $u_n$ lie on opposing central vertices of the graph, i.e. vertices $(\frac{a}{2}+1,\frac{b}{2}+1)$ and $(\frac{a}{2},\frac{b}{2})$ or $(\frac{a}{2},\frac{b}{2}+1)$ and $(\frac{a}{2}+1,\frac{b}{2})$. (Highlighted in Figure \ref{fig:even x even quadrants})
\item Vertices $u_{i}$ and $u_{i+1}$ alternate from quadrants 1 to 3 or from quadrants 2 to 4.
\end{enumerate}

The above conditions cannot be satisfied, as $u_1$ starts in quadrant 1, and there is no way to reach quadrants 2 or 4 without breaking condition (2), so $\max(d(s)) \neq \max(d_x(s)) + \max(d_y(s))$ as desired. Thus, we construct an ordering $s^*$ such that $d(s^*) = \max(d_x(s^*)) + \max(d_y(s^*)) - 1$. Starting with $u_1 = (\frac{a}{2}+1,\frac{b}{2}+1)$, we choose vertices alternating from quadrants 1 and 3, until all vertices in quadrants 1 and 3 have been chosen. Now we must break condition (2) by choosing the next vertex in either quadrant 2 or 4. Without loss of generality, we choose a vertex in quadrant 4, picking any vertex with a $y$-value of $\frac{b}{2}$. We then continue choosing vertices by alternating between quadrants 2 and 4 such that the final vertex, $u_n = (\frac{a}{2},\frac{b}{2}+1)$.

Considering the ordering $s^*$ we have just created, we see that $d_x(s^*) = \max(d_x(s))$, and $d_y(s^*) = \max(d_y(s)) - 1$ using $+$ sign $-$ sign analysis. In this case, we have a $+$ sign in Region A with $y$-value $\frac{b}{2}$, and two vertices with 1 $+$ sign in Region B with $y$-value $\frac{b}{2}+1$. Comparing this configuration to the one achieving maximum $y$-distance, we have moved a $+$ sign from the $\frac{b}{2}+1$ level to the $\frac{b}{2}$ level, resulting in a net decrease of 1. Thus, $t^+(G_{a,b}) = \max(d_x(s)) + \max(d_y(s)) - 1$.

\begin{figure}[tb]
    \centering
    \subfigure[Example $G_{8,6}$ split into four quadrants.]
    {
        \includegraphics[height=2.0in]{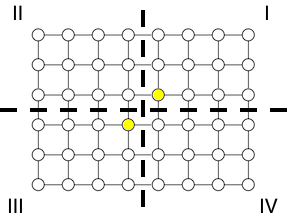}
        \label{fig:even x even quadrants}
    }
    \\
    \hspace*{\fill}
    \subfigure[Example $G_{8,7}$ split into four quadrants and a median in the odd dimension.]
    {
        \includegraphics[height=2.0in]{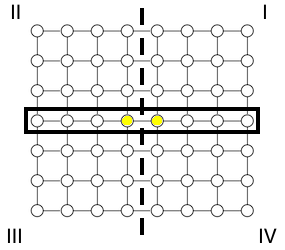}
        \label{fig:even x odd quadrants}
    }
    \hspace*{\fill}
    \subfigure[Example $G_{7,7}$ split into four quadrants and two medians.]
    {
        \includegraphics[height=2.0in]{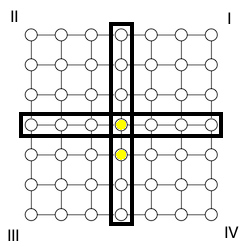}
        \label{fig:odd x odd quadrants}
    }
    \hspace*{\fill}
    \caption{Constructing $s^*$ that achieves $\max(d(s))$. In each case, start and end vertices of the ordering $u_1$ and $u_n$ are highlighted in yellow.}
    \label{fig:max xy distance}
\end{figure}

\textbf{Case B - $a$ is even, $b$ is odd:}
In this case, we are able to show $t^+(G_{a,b}) = \max(d_x(s)) + \max(d_y(s))$ when $a$ is even, $b$ is odd, and $a > 2$. Thus, there exists an ordering $s^*$ that satisfies the combined conditions below: 

\begin{enumerate}
\item $u_1$ and $u_n$ lie in the center of the $x$-median, i.e. $x_1 = (\frac{a}{2},\frac{b+1}{2})$ and $x_n = (\frac{a}{2} + 1,\frac{b+1}{2})$
\item There does not exist a pair $u_i, u_{i+1}$ such that $x_i,x_{i+1} \ge \frac{a}{2} + 1$, or $x_i,x_{i+1} \le \frac{a}{2}$, or $y_i,y_{i+1} > \frac{b+1}{2}$, or $y_i,y_{i+1} < \frac{b+1}{2}$.
\end{enumerate}

We construct $s^*$ by referring to quadrants shown in Figure \ref{fig:even x odd quadrants}. Start with $u_1 = (\frac{a}{2},\frac{b+1}{2})$, and choose $u_2$ in quadrant 1. Then choose vertices, alternating between quadrants 1 and 3, until all of the vertices in quadrants 1 and 3 have been used. Then, choose vertices on the $x$-median, choosing each successive vertex on opposite sides of the dotted line, such that $(\frac{a}{2}+1, \frac{b+1}{2})$ is unused (as it is reserved for $u_n$). Notice that this requires at least 4 vertices in the median section, explaining the condition $a > 2$ (and why ladder graphs are a special case). When all median vertices on the left side of the dotted line are chosen, choose a vertex in quadrant 4. Then, choose vertices alternating between quadrants 2 and 4 until all vertices in quadrant 2 have been used. We then choose the final vertex, $u_n = (\frac{a}{2}+1,\frac{b+1}{2})$. Since this ordering $s^*$ attains $\max(d_x(s))$ and $\max(d_y(s))$, we are done.

\textbf{Case C - $a,b$ are odd:}
First, we show that $\max(d(s)) \neq \max(d_x(s)) + \max(d_y(s))$.
Applying the results from Lemma \ref{lem:max x distance}, we split the grid into 4 quadrants and two medians, as shown in Figure \ref{fig:odd x odd quadrants} above.

Now, combining requirements to attain $\max(d_x(s))$ and $\max(d_y(s))$, each subsequent $u_{i+1}$ is under the following combined restrictions:
\begin{enumerate}
\item Vertices $u_1$ and $u_n$ reside in the intersection of the two medians.
\item There does not exist a pair $u_i, u_{i+1}$ such that $x_i,x_{i+1} > \frac{a+1}{2}$, or $x_i,x_{i+1} < \frac{a+1}{2}$, or $y_i,y_{i+1} > \frac{b+1}{2}$, or $y_i,y_{i+1} < \frac{b+1}{2}$.
\end{enumerate}

Clearly, condition (1) cannot be satisfied as there is only 1 vertex in the intersection of the medians. Thus, we have shown that $\max(d(s)) \neq \max(d_x(s)) + \max(d_y(s))$ as desired. We now construct $s^*$ such that $\max(d(s)) = \max(d_x(s)) + \max(d_y(s)) - 1$ by letting $u_1 = (\frac{a+1}{2}, \frac{b+1}{2})$ and $u_n = (\frac{a+1}{2}, \frac{b+1}{2} - 1)$, while following condition (2) above. Since $d_x(s^*) = \max(d_x(s))$, and we may verify that $d_y(s^*) = \max(d_y(s))-1$ using the same logic from Case A, we have $d(s^*) = \max(d_x(s)) + \max(d_y(s)) - 1$ as desired.



\section{Bumps on the Grid Graph}\label{sec:bump}
Bumps are the distinguishing factor between the upper traceable number problem and the radio number problem. Recall from Remark \ref{rem:additional value} that a bump $b_i$ is the additional value on $f(u_{i+1})$ than is required by its tightness neighbor constraint (Definition \ref{tightness neighbor}) from $u_i$.

A vertex $u_k$'s set of tightness neighbor constraints allows us to classify which vertex makes $b_k > 0$. Then, for the largest $c>1$ such that $u_{k-c} \in \tn(u_k)$, we say $u_k$ has a $(k-c)$ bump. 
We proceed to examine different types of $(k-c)$ bumps.


\textbf{The $(k-2)$ bump:}
Consider vertices $u_{k},u_{k-1},u_{k-2}$ for $3 \le k \le n$.

\begin{figure}[ht]
    \centering
    \hspace*{\fill}
    \subfigure[Example of the $(k-2)$ bump. 
    If we order vertices by increasing label, combined tightness constraints on $u_{k-2}, u_{k-1}$ and $u_{k-1}, u_{k}$ do not satisfy the tightness constraint on $u_k, u_{k-2}$ (in the rectangle).
    ]
    {
        \includegraphics[width=2.5in]{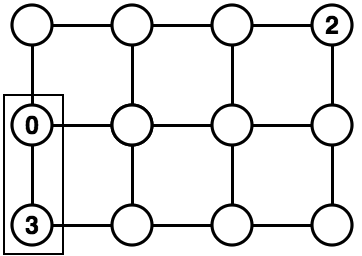}
        \label{fig:k-2 bump example}
    }
    \hspace*{\fill}
    \subfigure[Example of the bounding rectangle, where $d_\text{rect} = d(P,u_{k-1})$. In the  ordering $s$, we see $d(u_{k-2},u_{k-1}) = d(u_{k-2},P) + d(P, u_{k-1})$, and $d(u_{k-1},u_{k}) = d(u_{k-1},P) + d(P, u_{k})$. Thus we have $d_{k-1} + d_{k-2} - d(u_{k},u_{k-2}) = 2d_\text{rect}$.]
    {
        \includegraphics[width=2.5in]{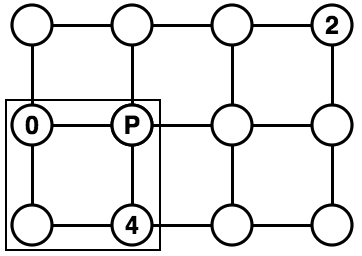}
        \label{fig:drect}
    }
    \hspace*{\fill}
    \caption{The $(k-2)$ bump on grid graphs.}
    \label{fig:clarify bump}
\end{figure}

Then, as shown in Figure \ref{fig:k-2 bump example}, in a $(k-2)$ bump on vertex $u_k$, we have
$$D+1 - d(u_{k},u_{k-2}) > f_{k-2}+f_{k-1} = (D+1 - d_{k-2} + b_{k-2}) + (D+1 - d_{k-1} + b_{k-1}).$$

Simplifying, we get $d_{k-1} + d_{k-2} - d(u_{k},u_{k-2}) > D+1 + b_{k-1} + b_{k-2}$. The left hand side expression $d_{k-1} + d_{k-2} - d(u_{k},u_{k-2})$ has an special interpretation on the grid graph, shown in Figure \ref{fig:drect}. Consider the bounding rectangle formed using vertices $u_{k-2}$ and $u_{k}$ as two corners. Then, define $d_\text{rect}$ to be the minimum distance from any vertex within the bounding rectangle to vertex $u_{k-1}$. Then we have $d_{k-1} + d_{k-2} - d(u_{k},u_{k-2}) = 2d_\text{rect}$. Substituting this into the inequality, we see that the following condition necessarily creates a $(k-2)$ bump:
\begin{equation}\label{eq:bump condition}
d_\text{rect} > \frac{D+1 + b_{k-1}}{2}.
\end{equation}

\textbf{The $(k-3)$ bump:}
There is no configuration of vertices $u_{k},u_{k-1},u_{k-2},u_{k-3}$ such that $b_i > 0$.
For sake of contradiction, let us assume that such a configuration exists. Then, we have 
$$D+1 - d(u_{k},u_{k-3}) > f_{k-1} + f_{k-2} + f_{k-3}.$$

The expression simplifies to $d_{k-1} + d_{k-2} + d_{k-2} - d(u_{k},u_{k-3}) > 2(D+1) + b_{k-1} + b_{k-2} + b_{k-3}$. By using a similar method to evaluate $d_{k-1} + d_{k-2} - d(u_{k},u_{k-2})$ on the $(k-2)$ bump, we see that $d_{k-1} + d_{k-2} + d_{k-2} - d(u_{k},u_{k-3}) \ge 2d(u_{k-1},u_{k-2})$. However, this is a contradiction because $2d(u_{k-1},u_{k-2}) > 2(D+1) + b_{k-1} + b_{k-2} + b_{k-3} > 2(D+1) \implies d(u_{k-1},u_{k-2}) > D+1$. Thus, our original assumption that there exists a configuration of $u_{k},u_{k-1},u_{k-2},u_{k-3}$ such that $b_i > 0$ was false.

\textbf{The $(k-q)$ bump, for $q \ge 4$:}
If there exists some $q$ such that 
\begin{equation}\label{eq:lastbump}
\min(f(u_k)-f(u_{k-q})) = \min\left(\sum_{i=1}^{q}f_{u-i}\right) > D,
\end{equation}
then no $u_{k-r}$, where $r \ge q$, is a tightness neighbor of $u_k$. With this, we prove that $s^*$ has no $(k-q)$ bumps, for $q \ge 4$.

\textbf{Case 1: there are no $(k-2)$ bumps.}
Let us consider $\min(f_{k-1}+f_{k-2})$, under the condition that there are no $(k-2)$ bumps, or $d_{k-1} + d_{k-2} - d(u_{k},u_{k-2}) \ge D+1$. Clearly, we want to maximize $d_{k-1} + d_{k-2}$ to minimize $f_{k-1}+f_{k-2}$ since $f_i = D + 1 - d_i$. Using the fact that $d_{k-1} + d_{k-2} - d(u_{k},u_{k-2}) = 2d_\text{rect}$, we see to maximize $d_{k-1} + d_{k-2}$ and have 0 bumps, we need $d_\text{rect} = \lfloor \frac{D+1}{2}\rfloor$. Then, 
$\max(d_{k-1} + d_{k-2}) = 2D - (D - \lfloor \frac{D+1}{2}\rfloor)$.
This yields $\min(f_{k-1}+f_{k-2}) = 2(D+1)-(D + \lfloor \frac{D+1}{2}\rfloor).$ Simplifying, we get 
$$\min(f_{k-1}+f_{k-2}) = \left\lceil\frac{D-1}{2}+2\right\rceil.$$
With this, we see that $\min(f_{k-1}+f_{k-2})+\min(f_{k-3}+f_{k-4}) = 2\lceil\frac{D-1}{2}+2\rceil > D$, so $q=4$ satisfies \eqref{eq:lastbump} and we have shown that there are no $(k-q)$ bumps for $q \ge 4$ in this case. 

\textbf{Case 2: there are $(k-2)$ bumps.}
If there exist $(k-2)$ bumps, then we have $d_{\text{rect}}>\lfloor \frac{D+1}{2} \rfloor$, so  $d(u_{k},u_{k-2}) < D-\lfloor\frac{D+1}{2}\rfloor$ since $d(u_k, u_{k-2}) \le D - d_{\text{rect}}$. Then, $f(u_{k}) - f(u_{k-2}) > D+1 - (D-\lfloor\frac{D+1}{2}\rfloor) = \lfloor\frac{D+1}{2} + 1\rfloor$. Again, we have $\min(f_{k-1}+f_{k-2})+\min(f_{k-3}+f_{k-4}) = \min(\sum_{i=1}^{4}f_{u-i}) > D$, so $q=4$ satisfies \eqref{eq:lastbump} and we are done.

From our work in this section, we observe the following:
\begin{remark}
On a grid graph $G_{a,b}$, all bumps must be $(k-2)$ bumps satisfying Equation \eqref{eq:bump condition}. We refer to this as the \textit{bump condition} on grid graphs.
\end{remark}

\section{Radio Number of Grid Graph}\label{sec:grid}
We return to the critical expression, $\max\left(\sum_{i=1}^{n-1}\left(d_i-b_i\right)\right)$ from Lemma \ref{lem:rn is d-b}. We solve for this expression in each subcase of the grid graph: even by even, even by odd, and odd by odd grids, giving the following result:
\begin{theorem}
On the grid graph $G_{a,b}$, with $a,b>2$,
$$
\rn(G_{a,b}) =
\begin{cases}
\dfrac{a^2b+b^2a}{2}-ab-a-b+6, & \normalfont\text{if }a,b\text{ are even}
\\[1em]
\dfrac{a^2b+b^2a-a}{2}-ab-b+2, & \normalfont\text{if }a\text{ is even, }b \text{ is odd}
\\[1em]
\dfrac{a^2b+b^2a-a-b}{2}-ab+2, & \normalfont\text{if }a,b\text{ are odd}
\end{cases}
$$
\end{theorem}
\begin{proof}
\textbf{Case A: $a$ is even, $b$ is odd} - We start with this case as it is the easiest and provides insight into other cases. Here, we can construct an ordering $s^*$ such that 
$$\max\left(\sum_{i=1}^{n-1}\left(d_i-b_i\right)\right) = t^+(G_{a,b}),$$
by constructing an ordering that attains $t^+(G_{a,b})$ as shown in the proof to Theorem \ref{thm:utn grid}, while ensuring there are no bumps.

\begin{figure}[tb]
    \centering
    \hspace*{\fill}
    \subfigure[Ordering $s^*$ on $G_{6,5}$. Parts (1), (2), and (3) are marked in red, green, and blue respectively. Notice the ordering of even vertices in quadrant 3.]
    {
        \includegraphics[height=1.8in]{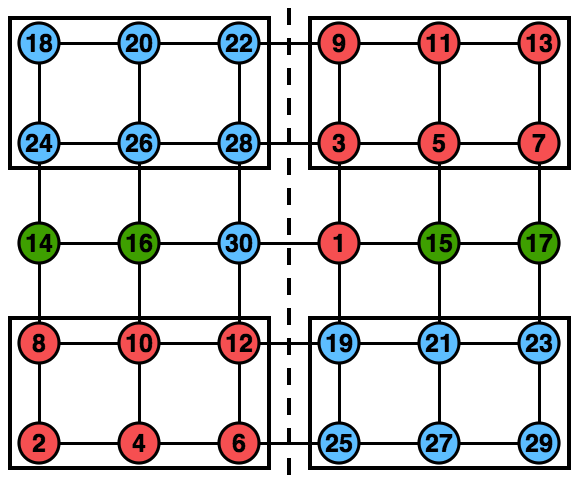}
        \label{fig:even x odd ordering}
    }
    \hspace*{\fill}
    \subfigure[Solution on $G_{6,5}$ using ordering $s^*$ from (a).]
    {
        \includegraphics[height=1.8in]{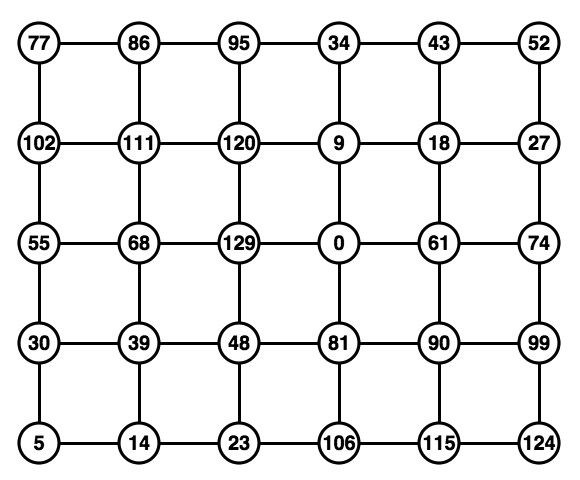}
        \label{fig:even x odd labeling}
    }
    \hspace*{\fill}
    \caption{Optimal ordering $s^*$ on $G_{6,5}$ and solution $f$ attaining $\rn(G_{6,5})$.}
    \label{fig:even x odd solution}
\end{figure}

We provide a three-part algorithm to create such an ordering $s^*$ (see Figure \ref{fig:even x odd ordering}).
\begin{enumerate}
\item Choose all vertices from quadrants 1 and 3,
\item Choose vertices in the median,
\item Choose all vertices from quadrants 2 and 4
\end{enumerate}
Before we define each step of the algorithm, for convenience we remap the vertices of $G_{a,b}$ in relation to the 4 quadrants and 1 median.
Let $q_{w,(x,y)}$ denote a vertex in quadrant $w$, with the position $(x,y)$ on the quadrant $w$ subgraph of $G_{a,b}$. For example, $(\frac{a}{2}+1,\frac{b+1}{2}+1) = q_{1,(1,1)}$, as it is the vertex $(1,1)$ in the quadrant 1 subgraph of $G_{a,b}$. Let $m_{z,p}$ denote a vertex on the median, where $z$ denotes either the left median $l$ or right median $r$, and $p$ denotes the position on the corresponding subgraph of $G_{a,b}$. For example, $(1,\frac{b+1}{2}) = m_{l,1}$ and $(\frac{a}{2}+1,\frac{b+1}{2}) = m_{r,1}$.

For part (1), we let $u_1 = m_{r,1}$ and $u_2 = q_{3,(1,1)}$. For each subsequent vertex $u_{i+1}$, we alternate between quadrants 1 and 3, such that if $u_{k} = q_{3,(x,y)}$, $u_{k+1} = q_{1,(x,y)}$. Clearly, the order of vertices chosen in quadrant 3 defines ordering $s^*$ for part (1), so we provide this ordering. As all these vertices are in quadrant 3, we use the shorthand $(x,y) = q_{3,(x,y)}$ below (See Figure \ref{fig:even x odd ordering}).
$$(1,1), (2,1), \ldots ,\left(\frac{a}{2},1\right), (1,2), (2,2), \ldots ,\left(\frac{a}{2},2\right), \ldots, \left(1,\frac{b-1}{2}\right), \left(2,\frac{b-1}{2}\right), \ldots ,\left(\frac{a}{2},\frac{b-1}{2}\right)$$  
Since each $d_{\text{rect}} = \frac{D+1}{2}$, there are no bumps in part (1).

After all vertices in quadrants 1 and 3 are chosen, we move to part (2). If the last vertex chosen in part (1) is $u_{k} = q_{1,(\frac{a}{2},\frac{b-1}{2})}$ for some $k$, let $u_{k+1} = m_{l,1}$. Then, we choose vertices alternating from left median to right median as follows: $m_{l,1}, m_{r,2}, m_{l,2}, m_{r,3}, \ldots, m_{l,\frac{a}{2}-1}, m_{r,\frac{a}{2}}$. Here, each $d_{\text{rect}} \leq \frac{a}{2}+1 < \frac{D+1}{2}$, so part (2) has no bumps.

By symmetry, part (3) is identical to part (1), so $s^*$ has no bumps and we are done. From Lemma \ref{lem:rn is d-b}, we have
$$\rn(G_{a,b}) = (ab-1)(a+b-1)-t^+(G_{a,b}) = \frac{a^2b+b^2a-a}{2}-ab-b+2.$$

\textbf{Case B: $a,b$ are even} - In this case, 
$\max\left(\sum_{i=1}^{n-1}\left(d_i-b_i\right)\right) = t^+(G_{a,b}) -2.$ To prove this, we construct such an ordering $s^*$, and then provide an upper bound.

\begin{figure}[tb]
    \centering
    \hspace*{\fill}
    \subfigure[Ordering $s^*$ on $G_{6,6}$. The beginning and end of ordering $s^*$ contain $(k-2)$ bumps and are in blue. The red vertices adhere to the ``diagonal method", while the green vertices adhere to a ``perturbed diagonal method". 
    ]
    {
        \includegraphics[height=2.15in]{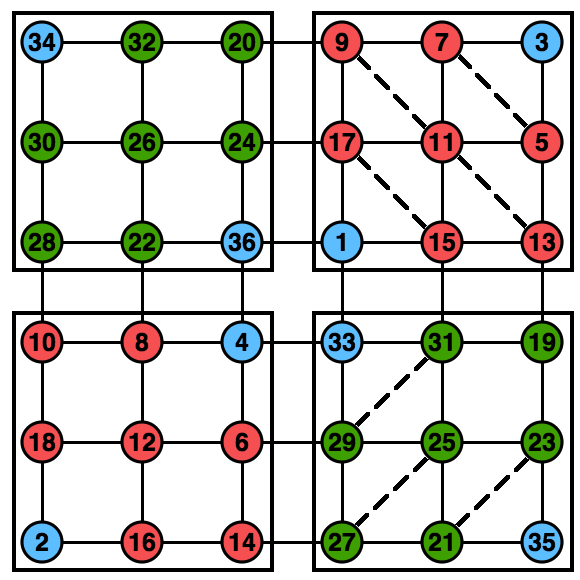}
        \label{fig:even x even ordering}
    }
    \hspace*{\fill}
    \subfigure[Solution on $G_{6,6}$ using ordering $s^*$ from (a).]
    {
        \includegraphics[height=2.15in]{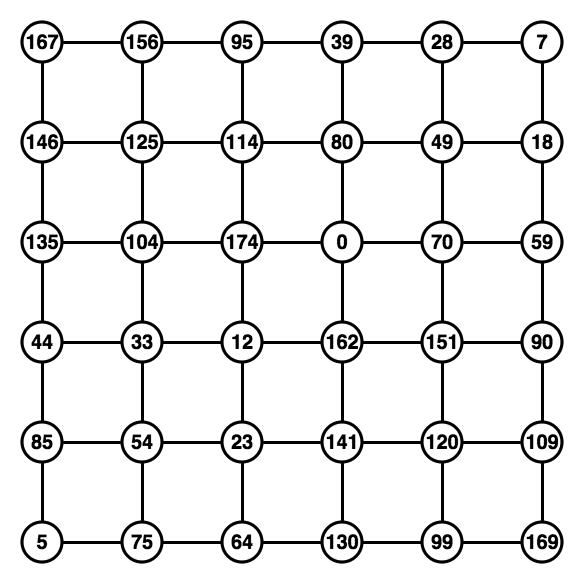}
        \label{fig:even x even labeling}
    }
    \hspace*{\fill}
    \caption{Optimal ordering $s^*$ on $G_{6,6}$ and solution $f$ attaining $\rn(G_{6,6})$.}
    \label{fig:even x even solution}
\end{figure}

To construct $s^*$ with $\sum_{i=1}^{n-1}(d_i-b_i) = t^+(G_{a,b}) -2$, we provide an algorithm using the same $q_{w,(x,y)}$ notation.
Let $u_1 = q_{1,(1,1)}, u_2 = q_{3,(1,1)}, u_3 = q_{1,(\frac{a}{2},\frac{b}{2})}, u_4 = q_{3,(\frac{a}{2},\frac{b}{2})}$. Notice that $d_{\text{rect}} = \frac{D+2}{2} > \frac{D+1}{2}$ for vertices $u_1,u_2,u_3$, so $b_2 = 1$. For vertices $u_2,u_3,u_4$, we have $d_{\text{rect}} = \frac{D+1+b_2}{2}$, so there is no bump.
Then for each subsequent $u_{i+1}$, if vertex $u_{i} = q_{1,(x,y)}$, let $u_{i+1} = q_{3,(x,y)}$. This time, the labeling scheme for quadrants 1 and 3 is completely determined by the ordering of vertices chosen in quadrant 1. By using a ``diagonal method" ordering for quadrant 1 shown by red vertices in Figure \ref{fig:even x even ordering}, we see that $d_{\text{rect}} = \frac{D+1}{2}$ for all $u_i$ and $ i \ge 3$.
If the last vertex chosen in quadrant 3 was $u_k$ for some $k$, let $u_{k+1} = q_{4,(\frac{a}{2},\frac{b}{2})}$. Again, each subsequent vertex alternates between quadrants 2 and 4, with each $u_{i+1}$ in quadrant 2 lying on the corresponding vertex $u_{i}$ from quadrant 4. Using a similar ``perturbed diagonal method" in green vertices in Figure \ref{fig:even x even ordering}, we choose all the remaining vertices other than $q_{2,(\frac{a}{2},1)},q_{2,(1,\frac{b}{2})},q_{4,(\frac{a}{2},1)},q_{4,(1,\frac{b}{2})}$. To close off, we let $u_{n-3}=q_{4,(1,\frac{b}{2})},u_{n-2}=q_{2,(1,\frac{b}{2})},u_{n-1}=q_{4,(\frac{a}{2},1)},u_{n}=q_{2,(\frac{a}{2},1)}$. By symmetry with $u_1,u_2,u_3$, we have $b_{n-2}=1$. Thus, the given ordering $s^*$ attains $t^+(G_{a,b})$ and has 2 bumps.

Now, we show $\max\left(\sum_{i=1}^{n-1}\left(d_i-b_i\right)\right) \leq t^+(G_{a,b}) -2.$ If we have an ordering $s^*$ such that $d(s^*) = t^+(G_{a,b})$, we must have at least 2 bumps, as $d_{\text{rect}} > \frac{D+1}{2}$ (satisfying bump the condition above) when we include each corner vertex $(1,1), (1,\frac{b}{2}),(\frac{a}{2},\frac{b}{2}),(\frac{a}{2},1)$ to ordering $s^*$, and every two corners induce at least 1 bump. The other case is an ordering $s^*$ with 0 bumps. 
Since including any corner of $G_{a,b}$ decreases $d(s^*)$ from $\max(d(s^*))$ by 1 and there are 4 corners, when there are 0 bumps, $d(s^*) \le t^+(G_{a,b}) - 4$. Thus, $\max(\sum_{i=1}^{n-1}\left(d_i-b_i\right)) \not> t^+(G_{a,b}) -2$ as desired.

With this, we have
$$\rn(G_{a,b}) = (ab-1)(a+b-1)-(t^+(G_{a,b}) -2)= \frac{a^2b+b^2a}{2}-ab-a-b+4.$$

\textbf{Case C: $a,b$ are odd} - In this case, 
$\max\left(\sum_{i=1}^{n-1}\left(d_i-b_i\right)\right) = t^+(G_{a,b}).$ We construct such an ordering $s^*$ and show that it has 0 bumps. The construction algorithm has 2 parts:
\begin{enumerate}
\item ``Expanding Square"
\item ``Sinking $\wedge$'s" (See Figure \ref{fig:odd x odd solution})
\end{enumerate}

\begin{figure}[tb]
    \centering
    \hspace*{\fill}
    \subfigure[Ordering $s^*$ on $G_{5,7}$. 
    The ``expanding square" consists of the red vertices, the ``sinking $\wedge$" is in dark green, and the ``sinking $\vee$" is in light green. The start and end vertices are in blue. We see there must be $\frac{b-a}{2}$ green $\wedge$'s.
    ]
    {
        \includegraphics[height=2.5in]{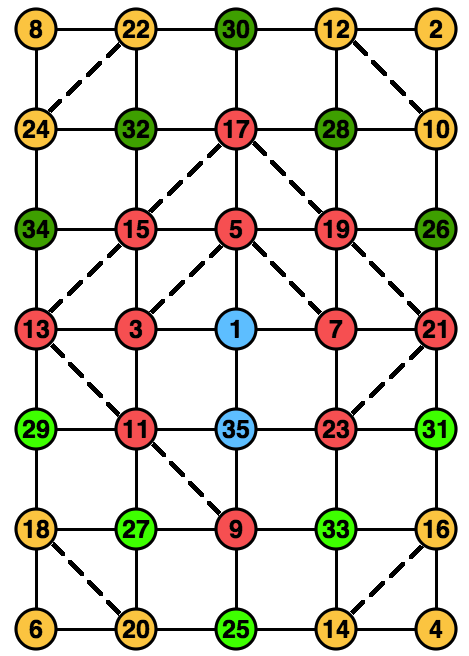}
        \label{fig:odd x odd ordering}
    }
    \hspace*{\fill}
    \subfigure[
    Solution on $G_{5,7}$ using ordering $s^*$ from (a).
    ]
    {
        \includegraphics[height=2.5in]{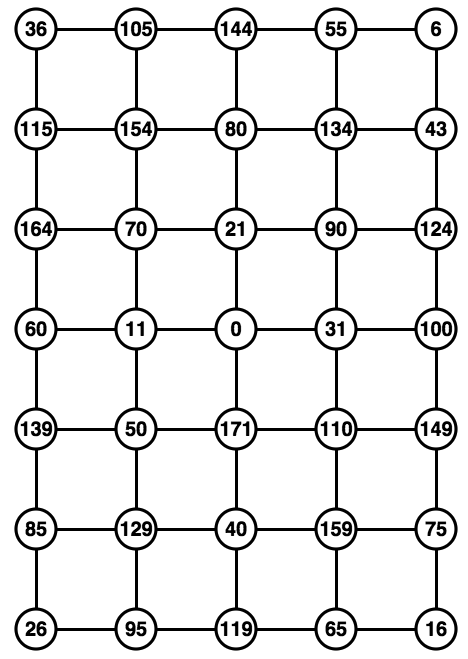}
        \label{fig:odd x odd labeling}
    }
    \hspace*{\fill}
    \caption{
    Optimal ordering $s^*$ on $G_{6,6}$ and solution $f^*$ attaining $\spn(f^*) = \rn(G_{5,7})$.}
    \label{fig:odd x odd solution}
\end{figure}

Throughout this proof, we concentrate on the odd vertices of the ordering, so for convenience $w_j = u_{2j-1}$. Then let $w_1 = (\frac{a+1}{2},\frac{b+1}{2}), w_2 = (\frac{a-1}{2},\frac{b+1}{2}), w_3 = (\frac{a+1}{2},\frac{b+3}{2}), w_4 = (\frac{a+3}{2},\frac{b+1}{2}), w_5 = (\frac{a+1}{2},\frac{b-3}{2})$. Notice that these 5 vertices create a square-like pattern about the center of the grid, except the final vertex is one vertex ``lower" than expected. As shown by red vertices in Figure \ref{fig:odd x odd ordering}, continue ``expanding" the square until all vertices in the $x$-median are chosen. The even vertices, $z_j = u_{2j}$ are chosen on the corners of the grid, with $z_1 = (a,b), z_2 = (a,1), z_3 = (1,1), z_4 = (1,b)$. As shown by yellow vertices in Figure \ref{fig:odd x odd ordering}, the even vertices are then distributed among the corners of the grid, such that $t^+(G_{a,b})$ and $d_{\text{rect}} < \frac{D+1}{2}$.

For part (2) of the algorithm, we concentrate on the even vertices of the ordering, $z_j$. As shown by dark green vertices in Figure \ref{fig:odd x odd ordering}, there are $\frac{b-a}{2}$ $\wedge$ shapes for which vertices have not been chosen. If the last vertex of the ``expanding square" was $w_k = (\frac{a+1}{2}, \frac{b-a}{2})$, let $z_k = (a, b - \frac{a-1}{2})$. Then, choose subsequent $z_{k+1}$ from each $\wedge$ travelling right to left, as shown in Figure \ref{fig:odd x odd ordering}. The corresponding $w_j$ creates a $\vee$, as these are the only vertices such that we attain $t^+(G_{a,b})$ and have 0 bumps. For all remaining $\wedge$'s, we ``sink" by repeating the same $\wedge$ and $\vee$ labeling procedure 1 $y$-coordinate below. Finally, when all $\wedge$'s and $\vee$'s are filled, we choose $u_n = (\frac{a+1}{2},\frac{b-1}{2})$. As $d(s^*) = t^+(G_{a,b})$ and there are no bumps on $s^*$, we have
$$\rn(G_{a,b}) = (ab-1)(a+b-1)-t^+(G_{a,b}) = \frac{a^2b+b^2a-a-b}{2}-ab+2.$$
\end{proof}

\section{Discussion}\label{sec:discussion}
In our paper, we established the exact values of the radio number not only for square grids with equal $x$ and $y$ dimension, but the general grid graph. This improves upon previous studies by Calles and Flores which found the upper and lower bounds $n^3-n^2-2n+4 \le \rn(G_{n,n}) \le n^3 - n^2 + 1$ for $n$ even and $n^3 - n^2 - n+2 \le rn(G_{n,n}) \le n^3 - n^2 - \frac{n-1}{2} + 1$ for $n$ odd on square grids. Our improved result stemmed from our understanding of $t^+(G_{a,b})$, the bump condition on grid graphs, and the subsequent search for orderings $s^*$ to maximize $\sum_{i=1}^{n-1}(d_i-b_i)$.

Whereas previous papers \cite{unpublished-grid, unpublished-ladder, m-ary trees, pathcycle, trees} used a ``sandwiching" method (strengthening upper and lower bounds of $\rn(G)$ until they coincide) to approach the radio number problem, our methodology simplifies the problem by recasting it as an optimization problem in orderings $s^*$. This perspective reduces the problem to considerations of distances $d_i$ that respect the bump condition instead of dealing with distances $d_i$, bumps $b_i$, and labels $f_i$ simultaneously. Furthermore, this methodology allows one to bypass the upper-bound lower-bound process, as finding an ordering $s^*$ that maximizes $\sum_{i=1}^{n-1}(d_i-b_i)$ immediately gives the result for $\rn(G)$.

Furthermore, our results on the grid graph (for example, Figures \ref{fig:even x odd labeling}, \ref{fig:even x even labeling}, and \ref{fig:odd x odd labeling}) provide network planners with a theoretically optimized model to appropriately space broadcast frequencies to most efficiently use a given bandwidth. These optimized frequency assignments should reduce channel collisions and increase the quality and efficiency of wireless networks. Moreover, our results should be more applicable to real-life situations than previous results on paths, cycles, and trees, as networks found in the real-world more closely resemble grids than path or tree structures.


\section{Conclusion}\label{sec:conclusion}
We have completely determined the radio number for all grid graphs $G_{a,b}$ where $a,b>2$, and characterized the $(k-2)$ bump condition on the grid graph. As a step in our proof, we also determined the upper traceable number on all grid graphs.

Using our methodology of determining the radio number of the grid graph, we anticipate extending our results to determine the radio numbers of $n$-dimensional grids, or the radio numbers of the Cartesian product between graphs. We can generalize our findings of the upper traceable number to grids of higher dimensions, and analyze $(k-c)$ bumps to extend the bump condition on these $n$-dimensional grids as well. Furthermore, the motif of isolating distances in the $x$ and $y$ directions can be applied to the Cartesian product of various graphs other than the path graph ($G_{a,b} = P_a \times P_b$ where $\times$ 
is the Cartesian product operator). We conjecture that on an $n$-dimensional grid graph, $t^+(G_{d_1,d_2,\ldots,d_n}) = \sum_{i=1}^{n}\max(d_{d_i}(s))$ if we can suitably place $u_1$ and $u_n$ and at least one of $d_i$ is odd. Several avenues of further research include investigating $n$-dimensional grid graphs, the Cartesian product of cycles, cellular graphs, and modified grid graphs.



\section{Acknowledgements}
I would like to thank my mentor, Dr. Dan Teague, for his guidance and support, and the North Carolina School of Science and Mathematics for sponsoring my research. I would also like to thank Dr. Myra Halpin, Dr. Jonathan Bennett, and Dr. Ming Ya Jiang for proofreading my paper, and Professor Drew Armstrong of the University of Miami and Mr. Tian-Yi Damien Jiang of the Massachusettes Institute of Technology for valuable advice on this paper.
\newpage

\bibliographystyle{amsplain}

\end{document}